\documentclass[11pt, reqno]{amsart}
\usepackage{amsmath, amsthm, amscd, amsfonts, amssymb, graphicx,mathrsfs, color}
\usepackage[colorlinks=true, linkcolor=blue,urlcolor=blue]{hyperref}
\input xy \xyoption{all}
\usepackage[all]{xy}
\usepackage{layout}
\usepackage{tikz}
\usepackage{bm}
\newtheorem{theorem}{Theorem}[section]

\newtheorem{proposition}[theorem]{Proposition}
\newtheorem{corollary}[theorem]{Corollary}
\theoremstyle{definition}

\theoremstyle{remark}
\numberwithin{equation}{section}

\newcommand{\N}{\mbox{$\mathbb{N}$}}

\newcommand{\Z}{\mbox{$\mathbb{Z}$}}

\begin{document}
\setcounter{page}{1}

\title[Graded quasi-Baer $\ast$-ring characterization of Steinberg algebras]{\textbf{Graded quasi-Baer $\ast$-ring characterization of Steinberg algebras}}

\author[M. Ahmadi, A. Moussavi]{Morteza Ahmadi, Ahmad Moussavi$^{*}$}
\address{Department of Pure Mathematics, Faculty of Mathematical Sciences\\
Tarbiat Modares University, P.O.Box:14115-134, Tehran, Iran}
\email{\textcolor[rgb]{0.00,0.00,0.84}{mortezy.ahmadi@modares.ac.ir\,\
		morteza.ahmadi23@gmail.com}}
\email{\textcolor[rgb]{0.00,0.00,0.84}{moussavi.a@modares.ac.ir\,\
 moussavi.a@gmail.com}}
\thanks{*Corrosponding authour: moussavi.a@gmail.com}

\begin{abstract}
Given a graded ample, Hausdorff groupoid $\mathcal{G}$, and an involutive field $K$, we consider the Steinberg algebra $A_K(\mathcal{G})$. 
We obtain necessary and sufficient conditions on $\mathcal{G}$ under which the annihilator of any graded ideal of $A_K(\mathcal{G})$ is generated by a homogeneous projection.  This property is called graded quasi-Baer $\ast$.
We use the Steinberg algebra model to characterize graded quasi-Baer $\ast$ Leavitt path algebras. 
\end{abstract}

\maketitle
\section{Introduction}
Steinberg algebras were independently introduced in \cite{Clark2014} and \cite{Steinberg2010}   and have attracted the attention of analysts and algebraists since then. Steinberg algebras appeared in the details of many groupoid C*-algebra constructions before they were specified by name (see, e.g., \cite{Exel2008, Kumijan1997}). In addition to providing insight into the analytic theory of groupoid C*-algebras, these algebras also gave rise to interesting examples of $\ast$-algebras. For instance, all Leavitt path algebras, Kumjian-Pask algebras, and discrete inverse semigroup algebras can be realized as Steinberg algebras. Furthermore, Steinberg algebras have been useful for the transfer of algebraic and analytic concepts and techniques.

In \cite{Clark2018} the authors characterized the graded ideals of Steinberg algebras over groupoids equipped~with~a~cocycle into a discrete group such that the inverse image of the identity doesn't have too much isotropy. 

In this paper, we study the annihilators of graded ideals in Steinberg algebras built from graded
groupoids. 
 We show that the annihilator of any graded ideal of $A_K(\mathcal{G})$ is generated by a homogeneous projection if and only if for each open invariant subset $U$ of the unit space $\mathcal{G}^{(0)}$, $U$ or the interior of $\mathcal{G}^{(0)}\setminus U$ is compact.
 This property (i.e.,  the annihilator of any graded ideal is generated by a homogeneous projection) is called graded quasi-Baer $\ast$. In \cite{ma2, dozd}, the authors characterized
graded quasi-Baer $\ast$ Leavitt path algebras. 
We give another characterization of graded quasi-Baer $\ast$ Leavitt path algebras by using the Steinberg algebra model.

\section{Quasi-Baer $\ast$ condition for graded unital $\ast$-rings}
 \subsection{Quasi-Baer $\ast$-rings}
For a subset $X$ of a ring $R$, 
 the right annihilator  $r_R (X)$ of $X$ in $R$ denotes the
set $\{r\in R\mid xr = 0\text{~ for ~all~}x \in  X\}$. 
It is straightforward to check that $r_R (X)$ is a right of $R$.


A ring $R$ is said to be a $\ast$-ring or an involutive ring, if it has an involution (i.e., an operation $\ast : R \rightarrow R$
such that $(x+ y)^{\ast} = x^{\ast} + y^{\ast}$, $(xy)^{\ast}= y^{\ast}x^{\ast}$, and $(x^{\ast})^{\ast}= x$ for all $x, y \in R$). 
A $\ast$-ring $R$ is
said to be a \textit{quasi-Baer $\ast$-ring} if $r_R (I)$ is generated by a projection for any ideal  $I$ of $R$.  
This condition is  left-right
symmetric. 
If a $\ast$-ring is quasi-Baer $\ast$, the projection which generates the right annihilator of zero is an identity. Consequently, quasi-Baer $\ast$-rings are necessarily unital.
 It turns out that the  projection in the definition of a quasi-Baer $\ast$-ring is central.
 
 Recall that an involution $\ast$ on a $\ast$-ring $R$ is said to be \textit{proper} if
$xx^{\ast}=0$ implies  $x= 0$
for any element $x \in R$. 
Also, an involution $ \ast $ is called a \textit{semiproper} involution if $xRx^{\ast} = 0$ implies $x = 0$. 
 Obviously, if $\ast$ is a proper involution, then $\ast$ is a semiproper involution. The converse
does not hold true (see \cite[Example 10.2.9]{Birkenmeier3}). 
It was shown in \cite[Lemma 10.2.10]{Birkenmeier3} that the involution on a quasi-Baer $\ast$-ring is always semiproper.

\subsection{Graded $\ast$-rings}
If $  \Gamma$ is an abelian group with identity $\varepsilon$, a ring $R$ is a $  \Gamma$-\textit{graded ring} if $R =\bigoplus_{g\in \Gamma}R_{g}$
such that each $ R_{g} $ is an additive
subgroup of $R$ and $R_{g} R_{h} \subseteq R_{gh}$ for all $g,h \in \Gamma$. The elements of $R^h =
\bigcup_{g\in \Gamma}R_{g}$
are the \textit{homogeneous elements} of $R$.   If $a \in R_{g}$
and $a \not= 0$, we say that $g $ is the degree of $a$.
 Note that  every nonzero homogeneous idempotent has degree $\varepsilon$.
If $R$ is an algebra over a field $K$, then $R$ is a \textit{graded algebra} if $R$ is a graded ring and $R_{g}$ is
a $K$-vector subspace for any $g\in \Gamma$.
A $ \Gamma $-graded ring $R$ with an involution $\ast $ is said to be a \textit{graded $\ast$-ring} if $R_{g}^{\ast}\subseteq R_{g^{-1}}$ for every
$g\in \Gamma$.

A \textit{graded right  ideal} of $R$ is a right  ideal $I$ such that $I =\bigoplus_{g\in\Gamma} I\cap R_{g}$. An  ideal $I$ of $R$ is a graded
 ideal if and only if $I$ is generated by homogeneous elements. This property implies that $r_R (X)$  is a
graded ideal of $R$ for any set $X$ of homogeneous elements of $R$. 
 
 
\subsection{Graded quasi-Baer $\ast$-rings}
In \cite{ma2},  the definition of quasi-Baer $\ast$-rings is adapted to graded $\ast$-rings. 
Recall that
a graded $\ast$-ring $R$ is called a \textit{graded quasi-Baer $\ast$-ring}
if the right  annihilator of any graded ideal of $R$ is generated by a homogeneous
projection.
It is useful to note that the homogeneous projection in the definition of a graded quasi-Baer $\ast$-ring is central (see \cite[Remark 2]{ma2}).

Recall that an involution $\ast$ on a graded $\ast$-ring $R$ is \textit{graded proper}, if $xx^*=0$ implies $x =0$ for any homogeneous element $x\in R$. Also, $\ast$  is called \textit{graded semiproper}, if $xRx ^* = 0$ implies $x = 0$ for any homogeneous element $x\in R$.
By \cite[Proposition 4]{ma2}, the involution on a graded quasi-Baer $\ast$-ring $R$ is always graded semiproper.
\section{Graded quasi-Baer $\ast$ condition for Steinberg algebras}
\subsection{Graded groupoids}
A \textit{groupoid} is a small category in which every morphism is invertible. It can
also be viewed as a generalisation of a group which has partial binary operation. Let $\mathcal{G}$  be a groupoid. The \textit{unit space} of $\mathcal{G}$ is the set
\[\mathcal{G}^{(0)}=\{\gamma\gamma^{-1}\mid \gamma\in \mathcal{G}\}= \{\gamma^{-1}\gamma\mid \gamma\in \mathcal{G}\}.\]
Groupoid \textit{source} and \textit{range} maps $s,r: \mathcal{G} \rightarrow\mathcal{G}^{(0)}$ are defined such that $s(\gamma) = \gamma^{-1}\gamma$ and $r(\gamma)=\gamma\gamma^{-1}$.
Elements of $\mathcal{G}^{(0)}$ are units in the sense that $\gamma s(\gamma)=\gamma$ and $r(\gamma)\gamma=\gamma$ for all $\gamma\in \mathcal{G}$.
For each $u\in \mathcal{G}^{(0)}$, the set $\mathcal{G}_u^u=s^{-1}(u)\cap r^{-1}(u)$ is a
group, called the \textit{isotropy group} based at $u$.
 The \textit{isotropy group bundle} of $\mathcal{G}$ is the set
\[ \mathrm{Iso}(\mathcal{G}) =\bigcup_{u\in \mathcal{G}^{(0)}}\mathcal{G}_u^u=\{\gamma\in \mathcal{G}\mid s(\gamma)=r(\gamma)\}.\]
  A subset $U \subseteq \mathcal{G}^{(0)}$ is \textit{invariant} if $s(\gamma) \in U$ implies $r(\gamma) \in U$. Equivalently, $U$ is invariant if $r(\gamma) \in U$ implies $s(\gamma) \in U$. Given such an invariant subset $U$, let
\[\mathcal{G}|_U=\{\gamma\in \mathcal{G}\mid s(\gamma)\in U\}.\]
Observe that $\mathcal{G}|_U$  is a subgroupoid of $  \mathcal{G}$. If, in addition, $U$ is an open subset of $\mathcal{G}^{(0)}  $, then $\mathcal{G}|_U$ is open in $\mathcal{G}$.

The  set of \textit{composable pairs} of  $\mathcal{G}$ is $\mathcal{G}^{(2)} = \{(\gamma,\alpha) \in \mathcal{G} \times \mathcal{G} | s(\gamma) = r(\alpha)\}$. For $U, V \subseteq \mathcal{G}$, we define
\[UV =\{\gamma\alpha \mid \gamma\in U, \alpha\in V, (\gamma,\alpha)\in \mathcal{G}^{(2)}\}.\]

A \textit{topological groupoid} is a groupoid endowed with a topology under which the inverse map is continuous, and such that composition is continuous with respect to the relative topology on $\mathcal{G}^{(2)}$ inherited from $\mathcal{G} \times \mathcal{G}$.
 An \textit{open bisection} of $\mathcal{G}$ is an open subset $U\subseteq \mathcal{G}$ such that $s|_U$ and $r|_U$ are homeomorphisms onto an open subset of $\mathcal{G}^{(0)}$.
An \textit{étale} groupoid is a topological groupoid
$\mathcal{G}$ such that its range map is a local homeomorphism from $\mathcal{G}$ to $\mathcal{G}^{(0)}$ (the source map will consequently share that property). It is easy to see that
the topology of an étale groupoid admits a basis formed by open bisections.
In an étale groupoid one has that $\mathcal{G}^{(0)}$ is open in $\mathcal{G}$. If, in addition, $\mathcal{G}$ is Hausdorff, then $\mathcal{G}^{(0)}$ is also closed in $\mathcal{G}$.
We say that an étale groupoid $\mathcal{G}$ is \textit{ample} if there is a basis
consisting of compact open bisections for its topology.

An ample hausdorff groupoid $\mathcal{G}$ is called  \textit{effective}  if $\textrm{Int(Iso}(\mathcal{G}))$, the interior of $\mathrm{Iso}(\mathcal{G})$ in the relative topology, is equal to  $\mathcal{G}^{(0)}$.
We  say that $\mathcal{G}$ is \textit{strongly effective} if $\mathcal{G}|_U$ is effective
for every closed invariant subset $U\subseteq \mathcal{G}^{(0)}$. If $\mathcal{G}$ is strongly effective, then it is effective because $\mathcal{G}^{(0)}$ is a closed invariant set.

Let $\Gamma$ be a discrete group with identity $\varepsilon$, and $\mathcal{G}$ a topological groupoid. A \textit{$\Gamma$-grading} of $\mathcal{G}$ is a continuous map $c: \mathcal{G}\rightarrow \Gamma$ such that $c(\gamma\alpha)= c(\gamma)c(\alpha)$ for all $(\gamma,\alpha)\in  \mathcal{G}^{(2)}$; such a map $c$ is called a \textit{cocycle} on $\mathcal{G}$. We always have $\mathcal{G}^{(0)}\subseteq c^{-1}(\varepsilon)$. Observe that 
for $g\in \Gamma$, $\textrm{Iso}(c^{-1}(g))=c^{-1}(g)\cap \textrm{Iso}(\mathcal{G})$.
We write $B^{co}_g(\mathcal{G})$
 for the collection of all  compact open bisections of $c^{-1}(g)$ and $B^{co}_*(\mathcal{G})=\bigcup_{g\in \Gamma}B^{co}_g(\mathcal{G})$.
Throughout this paper we only consider $\Gamma$-graded ample Hausdorff groupoids.

\subsection{Steinberg algebras}
 We recall the notion of the Steinberg algebra as a universal algebra generated by certain compact open subsets of an ample Hausdorff groupoid.
Let $ \mathcal{G} $ be a $\Gamma$-graded ample Hausdorff groupoid and $R$ be
a commutative ring with identity. The \textit{Steinberg $R$-algebra}
associated to $ \mathcal{G} $, denoted $A_R(\mathcal{G})$, is the algebra generated by the set 
$\{t_B\mid B\in B_*^{co}(\mathcal{G})\}$
 with coefficients in $R$, subject to
 \begin{enumerate}
 \item{$(R_1)$} $t_{\emptyset}=0$;
 \item{$(R_2)$} $t_{B_1}t_{B_2}=t_{B_1B_2}$ for all $B_1,B_2\in B_*^{co}(\mathcal{G})$; and
 \item{$(R_3)$} $t_{B_1}+t_{B_2}=t_{B_1\cup B_2}$, whenever $B_1$ and $B_2$ are disjoint elements of $B^{co}_g$
 for some $g\in \Gamma$ such that $B_1\cup B_2$ is a bisection.
 \end{enumerate}
 The Steinberg algebra defined above is isomorphic to the following construction:
\[A_R(\mathcal{G})= span\{1_U\mid U \text{~is a compact open bisection of~} \mathcal{G}\},\]
where $1_U: \mathcal{G}\rightarrow R$ denotes the characteristic function on $U$ (see \cite[Theorem 3.10]{Clark2014}). Equivalently, if we give $R$ the discrete topology, then continuous functions from $\mathcal{G}$ to $R$ are exactly locally constant functions from $ \mathcal{G} $ to $R$, and so $A_R(\mathcal{G})=C_c(\mathcal{G},R)$, the space of compactly supported continuous functions from $ \mathcal{G} $ to $R$. Addition is
point-wise and multiplication is given by convolution
$(f*g)(\gamma)=\sum_{\alpha\beta=\gamma}f(\alpha)g(\beta)$.
It is useful to note that
$1_U * 1_V=1_{UV}$ for compact open bisections $U$ and $V$.
By \cite[Lemma 3.5]{Clark2014}, every element $f \in A_R(\mathcal{G})$ can be expressed as $f=\sum_{U\in F}a_U1_U$, where $F$ is a finite subset of mutually disjoint elements of $ B_*^{co}(\mathcal{G}) $.

The family of all idempotent elements of $A_R(\mathcal{G}^{(0)})$ is a set of local units for $A_R(\mathcal{G})$. Moreover, $A_R(\mathcal{G})$ is unital if and only if $\mathcal{G}^{(0)}$ is compact. In this case, $1_{\mathcal{G}^{(0)}}$ is the identity element of $A_R(\mathcal{G})$.

If  $c: \mathcal{G}\rightarrow \Gamma$ is a cocycle, 
then the Steinberg algebra $A_R(\mathcal{G})$  is a $\Gamma$-graded algebra with homogeneous components
\[A_R(\mathcal{G})_g=\{f\in A_R(\mathcal{G}) \mid f(\gamma) \not= 0\Rightarrow c(\gamma)=g\}.\]
If $^-:R\rightarrow R$ is an involution on $R$, then the map 
\[*:A_R(\mathcal{G})\rightarrow A_R(\mathcal{G}),\hspace*{2cm}f\mapsto f^*, ~where ~f^*(\gamma)=\overline{f(\gamma^{-1})}\]
defines an involution on $A_R(\mathcal{G})$ making it into a $\ast$-algebra.
Observe that $(A_R(\mathcal{G})_g)^*=A_R(\mathcal{G})_{g^{-1}}$, for each $g\in \Gamma$.
It follows that $A_R(\mathcal{G})$ is a $\Gamma$-graded $\ast$-algebra.

A function $f \in A_R(\mathcal{G})$ is a \textit{class function} if $f$ satisfies the following conditions:
\begin{enumerate}
\item $f(x) \not= 0\Rightarrow s(x) = r(x)$;
\item $ s(x) = r(x) = s(z) \Rightarrow f(zxz^{-1}) = f(x) $.
\end{enumerate}
By \cite[Proposition 4.13]{Steinberg2010} the center of $ A_R(\mathcal{G})$ is the set of class functions.

The following result  is a characterization of Steinberg algebras that have
graded proper involutions.
\begin{proposition}\label{pro11}
Let $R$ be a commutative unital $\ast$-ring, 
$\mathcal{G}$ be an ample Hausdorff groupoid, $\Gamma$ be a discrete group, and $c :\mathcal{G}\rightarrow \Gamma$ be a cocycle such that $c^{-1}(\varepsilon)$ is effective.
Then the following are equivalent.
\begin{enumerate}
\item\label{pro11_1} The involution on $R$ is proper;
\item\label{pro11_2} The involution on $A_R(\mathcal{G})$ is graded proper;
\item\label{pro11_3} The involution on $A_R(\mathcal{G})$ is graded semiproper.
\end{enumerate}
In particular, if $K$ is a field with involution, then the involution on $A_K(\mathcal{G})$ is graded proper.
\end{proposition}
\begin{proof}
\ref{pro11_1}$\Rightarrow$\ref{pro11_2}
 Assume on the contrary that there exists  a nonzero homogeneous element $f\in A_R(\mathcal{G})_{g}$
such that $ff^*=0$. 
We can express $f$  as $f=\sum_{i=1}^nr_{i}1_{U_i}$,  where $r_1,\ldots,r_n\in R\setminus\{0\}$  and $U_1,\ldots,U_n \in B^{co}_*(\mathcal{G})$ are mutually disjoint.
Since the $U_i$'s are disjoint and the $r_i$'s are nonzero, we can assume each $U_i\subseteq \mathcal{G}_{g}$. Without loss of generality, we can assume that  $r(U_1),\ldots,r(U_n) $ are mutually disjoint compact open subsets of $\mathcal{G}^{(0)}$. For this, let $W=r(U_i)\cap r(U_j)\not=\emptyset$, for some $i\not=j$. Then $W$ is a nonempty compact open subset of $\mathcal{G}^{(0)}$. Let $x\in W$, then there exist $\alpha\in U_i$ and $\beta\in U_j$ such that $x=r(\alpha)=r(\beta)$. So $\alpha=\beta(\beta^{-1}\alpha)$, and that $\beta^{-1}\alpha\in U_j^{-1}U_i$.  It is easy to see that $U_j^{-1}U_i$ does not intercept $\mathcal{G}^{(0)}$. Then
$U_j^{-1}U_i\subseteq c^{-1}(\varepsilon)\setminus \mathcal{G}^{(0)}$ is a nonempty compact open bisection. Since $ c^{-1}(\varepsilon) $ is effective,  \cite[Lemma 3.1]{Brown2014} implies that there exists a nonempty open
subset $V \subseteq W$ such that $V(U_j^{-1}U_i)V = \emptyset$. 
By shrinking if necessary, we can assume $V$ is compact.
Then $U_jV$ and $U_iV$ are mutually disjoint compact open bisection, and
$(U_jV)^{-1}(U_iV)= \emptyset$. This implies $r(U_jV)\cap r(U_iV)=\emptyset$. If
we define  $f'=1_V*f$, then $f'\not=0$, $f'f'^*=0$, and  $f'$  can be expressed  as $f'=\sum_{i=1}^nr_{i}1_{U'_i}$,  where $r_1,\ldots,r_n\in R\setminus\{0\}$, $U'_1,\ldots,U'_n \in B^{co}_*(\mathcal{G})$ are mutually disjoint, and $r(U'_1),\ldots,r(U'_n) \subseteq \mathcal{G}^{(0)}$ are mutually disjoint as desired. 

Next, we have
\[ff^*=\sum_{i,j=1}^nr_i\overline{r_j}(1_{U_i}*1_{U_j^{-1}})=\sum_{i,j=1}^n r_{i}\overline{r_{j}}1_{U_iU_j^{-1}}=0.\]
 For each $i$ if $j\not=i$, then $U_iU_j^{-1}$ does not intercept $\mathcal{G}^{(0)}$.
Thus 
\begin{align*}
 ff^*|_{\mathcal{G}^{(0)}}&=\sum_{i,j=1}^n r_i\overline{r_j}(1_{U_iU_j^{-1}})|_{\mathcal{G}^{(0)}}=\sum_{i,j=1}^n r_i\overline{r_j}1_{U_iU_j^{-1}\cap\mathcal{G}^{(0)}}\\&=\sum_{i=1}^n r_i\overline{r_i}1_{U_iU_i^{-1}}=\sum_{i=1}^n r_i\overline{r_i}1_{r(U_i)}=0.
 \end{align*}
Note that any collection of characteristic functions of mutually disjoint open compact subsets of $\mathcal{G}^{(0)}$ is linearly independent. 
Then $ r_i\overline{r_i}=0$ for each $i$. Thus, \ref{pro11_1} yields $r_1,\ldots,r_n=0$
and that $f=0$, a contradiction. Hence the involution on $A_R(\mathcal{G})$ is
graded proper.

 \ref{pro11_2}$ \Rightarrow $\ref{pro11_3}  This implication holds in any $\ast$-ring with local units.

\ref{pro11_3}$ \Rightarrow $\ref{pro11_1}  Assume that $a\overline{a}=0 $, for $a\in R$. Then for a nonempty open compact subset $U\subseteq \mathcal{G}^{(0)}$ we have $(a1_U)A_R(\mathcal{G})(a1_U)^* =(a1_U)A_R(\mathcal{G})(\overline{a}1_U)=(a\overline{a}1_U)A_R(\mathcal{G})(1_U)=0$.
Thus, \ref{pro11_3} yields that $a1_U=0$. Hence $a=0$ as desired.
\end{proof}
An argument similar to that used in the proof of Proposition \ref{pro11} can be used to prove the following.
\begin{proposition}\label{}
Let $R$ be a $\ast$-ring, and  $\mathcal{G}$ be an affective ample Hausdorff groupoid.
Then the following are equivalent.
\begin{enumerate}
\item\label{pro11_1} The involution on $R$ is proper;
\item\label{pro11_2} The involution on $A_R(\mathcal{G})$ is  proper;
\item\label{pro11_3} The involution on $A_R(\mathcal{G})$ is  semiproper.
\end{enumerate}
In particular, if $K$ is a field with involution, then the involution on $A_K(\mathcal{G})$ is  proper.
\end{proposition}
In the next result, we  characterize the graded quasi-Baer $\ast$  Steinberg algebras
over ample Hausdorff groupoids equipped with a cocycle taking values in a discrete group.

\begin{theorem}\label{1}
Let $K$ be a field with involution, $\mathcal{G}$ be an ample Hausdorff groupoid with compact unit space, $\Gamma$ be a discrete group, and $c :\mathcal{G}\rightarrow \Gamma$ be a cocycle such that $c^{-1}(\varepsilon)$ is strongly effective. 
 Then $A_K(\mathcal{G})$ is a graded quasi-Baer $\ast$-ring if and only if for every open invariant subset $U$ of $\mathcal{G}^{(0)}$, $U$ or $\mathrm{Int}(\mathcal{G}^{(0)}\setminus U)$ is compact.
\end{theorem}
\begin{proof}
Let $I$ be a graded ideal of $A_K(\mathcal{G})$. Then by \cite[Theorem 5.3]{Clark2018}, there exists an open invariant subset $U$ of $\mathcal{G}^{(0)}$ such that $I=A_K(\mathcal{G}|_U)$. 
First, assume that $U$ is compact. Then by \cite[Lemma 1.6]{Clark2019}, $1_U$ is a class function in $A_K(\mathcal{G})$, and by \cite[Proposition~4.13]{Steinberg2010}, $1_U$ is in the center of $A_K(\mathcal{G})$. Thus $1_U$ is a central projection. We claim that
$I=A_K(\mathcal{G})1_U$.
 Let $B$ be a compact open bisection in $\mathcal{G}$.   
 Then 
 $BU=B$ if $s(B)\subseteq U$, and $BU= \emptyset$ otherwise.
 So  $1_B*1_U=1_{BU}\in A_K(\mathcal{G}|_U)=I$.
Since $A_K(\mathcal{G})$ is spanned by the elements of the form $1_B$, with $B$ as above, we conclude that $A_K(\mathcal{G})1_U\subseteq I$. Let us now prove that $I\subseteq A_K(\mathcal{G})1_U$. Again we may focus on characteristic
functions, meaning that all we must do is show that $1_B$ lies in $A_K(\mathcal{G})1_U$, for all compact open bisections $B\subseteq \mathcal{G}|_U$. Given such a $B$, observe that
$s(B)\subseteq U$, then  \[1_B=1_{Bs(B)}=1_{BU}=1_B*1_U\in A_K(\mathcal{G})1_U,\]
This prove that $I= A_K(\mathcal{G})1_U$. It is straightforward to see that
$r_{A_K(\mathcal{G})}(A_K(\mathcal{G})1_U)=(1_{\mathcal{G}^{(0)}}-1_U)A_K(\mathcal{G})$. Thus $r_{A_K(\mathcal{G})}(I)$ is generated by  $1_{\mathcal{G}^{(0)}}-1_U$.

Now, assume that $V=\mathrm{Int}(\mathcal{G}^{(0)}\setminus U)$ is compact.
Observe that $V$ is also invariant. Indeed, $r(s^{-1}(V))$ is an open subset of $\mathcal{G}^{(0)}$, since $V$ is an open subset of $\mathcal{G}^{(0)}$. Notice that $\mathcal{G}^{(0)}\setminus U$ is invariant, since $U$ is invariant. Since $\mathcal{G}^{(0)}\setminus U$ is invariant, $r(s^{-1}(V))\subseteq \mathcal{G}^{(0)}\setminus U$. Thus $r(s^{-1}(V))\subseteq V$, and that $V$ is invariant. Hence $1_V$ is a central (homogeneous) projection. We claim that the right annihilator of $I$ is generated by $1_V$.
Let $1_B\in I$, for some compact open bisection $B\subseteq \mathcal{G}|_U$. Then $s(B)\subseteq U$, and so $s(B)\cap V\subseteq s(B) \cap(\mathcal{G}^{(0)}\setminus U)=\emptyset$. Thus $BV=\emptyset$, and that $1_B*1_V=1_{BV}=0$. Hence
$1_VA_K(\mathcal{G})\subseteq r_{A_K(\mathcal{G})}(I)$. Now, assume that $1_B\in r_{A_K(\mathcal{G})}(I)$, for some compact open bisection $B\subseteq \mathcal{G}$. 
Then for each $1_W\in I$, $1_W1_B=1_{WB}=0$. Since $s(W)\subseteq U$, we have $r(B)\subseteq \mathcal{G}^{(0)}\setminus U$. Since $r(B)$ is open,  $r(B)\subseteq V$. Then 
$1_B=1_{Br(B)}=1_{BV}=1_B1_V=1_V1_B\in 1_VA_K(\mathcal{G})$. This show that $r_{A_K(\mathcal{G})}(I)\subseteq  1_VA_K(\mathcal{G})$. Then  $r_{A_K(\mathcal{G})}(I)$ is generated by $1_V$ as claimed. 

Conversely,  assume that $A_K(\mathcal{G})$ is a graded quasi-Baer $\ast$-ring. Let $U$ be an open invariant subset of $\mathcal{G}^{(0)}$. Take $I=A_K(\mathcal{G}|_U)$. By \cite[Theorem 5.3]{Clark2019} $I$ is a graded ideal of $A_K(\mathcal{G})$. Then there exists a central homogeneous  projection $p\in A_K(\mathcal{G})$ such that the right annihilator of $I$ is generated by $p$. First, assume that $p=0$. If there is a compact open set
$V\subseteq  \mathrm{Int}(\mathcal{G}^{(0)}\setminus U)$, then one can show that $1_V\in r_{A_K(\mathcal{G})}(I)$, a contradiction. Thus $ \mathrm{Int}(\mathcal{G}^{(0)}\setminus U=\emptyset) $
is  compact open invariant. Now, assume
that $p\not=0$. We claim that $U$ or $ \mathrm{Int}(\mathcal{G}^{(0)}\setminus U)$ is compact.
First, note that  \cite[Lemma 1.3]{Clark2018} implies that there exists $\{k_i\}_{i=1}^t\subseteq K\setminus\{0\}$
such that
 $p=\sum_{i=1}^tk_i1_{U_i}$ where $U_i=p^{-1}(k_i) $ is a compact open bisection for each $i$.
 Since $p$ is central, it is a class function, so $1_{U_i}$ is also a class function for each $i$.
It is well-known that every homogeneous idempotent is in the zero component. Then by \cite[Lemma 1.4]{Clark2018},  $U_i$ is a compact open invariant set for each $i$. Also, it is clear that $U_1,U_2,\ldots,U_t$ are mutually disjoint.
 Assume on the contrary that $U$ and $ \mathrm{Int}(\mathcal{G}^{(0)}\setminus U)$ are not compact.  Then $U_i\not=U$ and $U_i\not=\mathrm{Int}(\mathcal{G}^{(0)}\setminus U)$ for each $i$.
If  $U_j\cap U\not=\emptyset$ for some $j$, then there is a compact open subset $V$ of $\mathcal{G}^{(0)}$ such that $V\subseteq U_j\cap U$. Since $\mathcal{G}^{(0)}$ is open in $G$, $V$ is a compact open bisection in $\mathcal{G}$, so $1_V\in A_K(\mathcal{G})$. Since $s(1_V)=V\subseteq U$, $1_V\in I$.
But, 
\[1_V*p=\sum_{i=1}^tk_i(1_V*1_{U_i})=k_j1_V+\sum_{\substack{i=1\\i\not=j}}^tk_i1_{VU_i}=k_j1_V\not=0,\] a contradiction. Thus $U_j\cap U=\emptyset$ and so $U_j\subsetneq \mathrm{Int}(\mathcal{G}^{(0)}\setminus U)$ for each $i$. Since $\cup_{i=1}^t U_i$ is compact and $\mathrm{Int}(\mathcal{G}^{(0)}\setminus U)$ is not compact, 
$\bigcup_{i=1}^t U_i\subsetneq \mathrm{Int}(\mathcal{G}^{(0)}\setminus U)$, and that
$\mathrm{Int}(\mathcal{G}^{(0)}\setminus U)\setminus \cup_{i=1}^t U_i\not=0$.
Then there is a compact open subset $V$ of $\mathcal{G}^{(0)}$ such that $V\subseteq \mathrm{Int}(\mathcal{G}^{(0)}\setminus U)\setminus \cup_{i=1}^t U_i$. As discussed above, $V$ is a compact open bisection in $\mathcal{G}$, so $1_V\in A_K(\mathcal{G})$. 
One can show that $1_V\in r_{A_K(\mathcal{G})}(I)=pA_K(\mathcal{G})$. But,
\[p*1_V=\sum_{i=1}^tk_i(1_{U_i}*1_V)=\sum_{i=1}^tk_i1_{U_iV}=0\not=1_V,\]
a contradiction. Therefore, $U$ or $\mathrm{Int}(\mathcal{G}^{(0)}\setminus U)$ must be compact.
\end{proof}
It should be noted that the involution on a graded quasi-Baer $\ast$-ring must be graded semiproper. However, we don't need any additional assumption on the field $K$  in the previous theorem. Indeed, Proposition \ref{pro11} guarantees the involution on $A_K(\mathcal{G})$ is graded semiproper.

As a consequence of Theorem \ref{1}, we have the following.
\begin{corollary}\label{}
Let $K$ be a field with involution, $\mathcal{G}$ be an ample Hausdorff groupoid with compact unit space, $\Gamma$ be a discrete group, and $c :\mathcal{G}\rightarrow \Gamma$ be a cocycle such that $c^{-1}(\varepsilon)$ is strongly effective. 
If $\mathcal{G} $ is strongly effective, then $A_K(\mathcal{G})$ is a quasi-Baer $\ast$-ring if and only if for every open invariant subset $U$ of $\mathcal{G}^{(0)}$, $U$ or $\mathrm{Int}(\mathcal{G}^{(0)}\setminus U)$ is compact.
\end{corollary}
\begin{proof}
Note that if $\mathcal{G} $ is strongly effective, then every ideal of $A_K(\mathcal{G})$ is graded. For this, let $I$ be an ideal of $A_K(\mathcal{G})$. Then by \cite[Corollary 3.7]{Clark2019-2} there is an open invariant subset $U\subseteq \mathcal{G}^{(0)}$ such that $I=A_K(\mathcal{G}|_U)$.
If $f=\sum_{B\in F}a_B1_B\in I$, where $F$ is a finite subset of mutually disjoint elements of $B^{co}_*(\mathcal{G})$. It is then clear that $1_B$ is also in $A_K(\mathcal{G}|_U)=I$ for each $B\in F$, so $I$ is indeed a graded ideal.

Then we can see that $A_K(\mathcal{G})$ is  quasi-Baer $\ast$ if and only if it is graded quasi-Baer $\ast$. So the result is now a direct consequence
of Theorem \ref{1}.
\end{proof}
\section{Graded quasi-Baer $\ast$ condition for Leavitt path algebras}
In this section we explain what Proposition \ref{pro11} and Theorem \ref{1}
say about a Leavitt path algebra of a directed graph. We start by gathering background needed to
state Corollaries \ref{8.2} and \ref{8.4}.
\subsection{Graph concepts}
A directed graph $E = (E ^0 ,E ^1 ,r,s)$ consists of two sets $E ^0$ , $E^ 1$ and two maps $ r ,s : E ^1~\rightarrow~E ^0 $. The elements of $E ^0$ are called \textit{vertices} and the elements of $E ^1$ \textit{edges}. If $s^{ -1} (v)$ is
a finite set for every $v \in E ^0 $, then the graph is called \textit{row-finite}. In this setting, if the number of vertices is finite, then the number of edges is finite as well, and we call $E$ a \textit{finite} graph.

A vertex $v$ for which $s ^{-1} (v)$ is empty is called a \textit{sink}, while a vertex $w$ for which $r ^{-1} (w)$
is empty is called a \textit{source}. A vertex $v \in E ^0$ such that $|s ^{-1} (v)| = \infty$, is called an \textit{infinite
emitter}. If $v$ is either a sink or an infinite emitter, then it is called a \textit{singular vertex}. If $v$
is not a singular vertex, it is called a \textit{regular vertex}.
The expressions $\textrm{Sink}(E)$, $\textrm{Source}(E)$,  $\textrm{Reg}(E)$, and $\textrm{Inf}(E)$ will be
used to denote, respectively, the sets of sinks, sources, regular vertices, and infinite
emitters of $E$.

 A finite path $\mu$ in a graph $E$ is a
sequence of edges $\mu = \mu_ 1\ldots \mu_ k $  such that $r(\mu_i ) = s(\mu_{ i+1} )$, $1 \leq i \leq k - 1$. In this case,
$s(\mu) := s(\mu_ 1 )$ is the \textit{source} of $\mu$, $r(\mu) := r(\mu_ k )$ is the \textit{range} of $ \mu $, and $k$ is the \textit{length} of
$ \mu $ which is denoted by $|\mu|$. 
The set of all finite paths of a $E$ is denoted by $\textrm{Path}(E)$.
An infinite path is an infinite sequence of edges $\mu = \mu_ 1\mu_ k \ldots$ such that $r(\mu_i ) = s(\mu_{ i+1} )$  for every $i \in \N$. The set of all infinite paths of  $E$ is denoted by $E^{\infty}$.

A \textit{preorder} $\geq$ on $E^0$ defined by:
$v \geq w$ if there is a path $\mu \in \textrm{Path}(E)$ such that $s(\mu)=v$ and  $r(\mu)=w$.
If $v \in E^0$ then the
\textit{tree} of $v$ is the set
$T(v)=\{w  \mid w\in E^0,v\geq w\}.$
A subset $H$ of $E^ 0$ is called \textit{hereditary} if $v \geq w$ and $v \in H$
imply $w \in H$. A hereditary set is \textit{saturated} if every regular vertex which
feeds into $H$ and only into $H$ is again in $H$, that is, if $s^{ -1} (v) \not= \emptyset$ is finite and
$r(s^{ -1} (v)) \subseteq H$ imply $v \in H$. For a hereditary subset $H$ we denote by $\overline{H}$, the
saturated closure of $H$, i.e., the smallest hereditary and saturated subset of $E^ 0$
containing $H$.

Let $H$ be a nonempty hereditary
and saturated subset of $E^ 0$. 

Following \cite{Clark2019}
we define
\begin{align*}
	F _E (H):=&\{\alpha\in \operatorname{Path}(E) \mid s(\alpha_ 1 ), r(\alpha_ i )\in E^ 0\setminus  H \text{ for } i < |\alpha|, r(\alpha_{ |\alpha|} )\in H\}.
\end{align*}
Given a nonempty hereditary and saturated subset $H$ of $E^0$, define
\[H^{tc}=\{u\in E^0\mid H\cap T(u)=\emptyset\}.\] 
The subset $H^{tc}$ is  hereditary and saturated, and $H\cap H^{tc}=\emptyset$ (see \cite[Remark 13]{ma2}).
We note that $H^{tc}$ corresponds to $H'$ from \cite{4}, $E^0-\overline{H}$ from \cite{3}, and $H^{\perp}$ from \cite{dozd}.
\subsection{Leavitt path algebras}
For a graph $E$ and a field $K$,
the Leavitt path algebra of $E$, denoted by $L _K (E)$, is the algebra generated by the sets
$\{v \mid v \in E^ 0 \}$, $\{e \mid e \in E^ 1 \}$, and $\{e^* \mid e \in E^ 1 \}$ with the coefficients in $K$, subject to the
relations
\begin{enumerate}
\item[(V)] $v_iv_j=\delta_{ij}v_i$  for every $v_i,v_j\in  E^ 0$,
\item[(E1)] $s(e)e=e r(e)=e$   for all $e\in  E^ 1 $,
\item[(E2)] $r(e)e^*=e^* s(e)=e^*$   for all $e\in  E^ 1 $,
\item[(CK1)] $e^*e'=\delta_{ee'}r(e)$   for all $e,e'\in E ^1$,
\item[(CK2)] $\sum_{\{e\in E ^1 ,s(e)=v\}}ee^*= v$, for every  vertex $v \in Reg(E) $.
\end{enumerate}

 It can be proved that $L_K(E)$ is a unital ring  if and only if  $E^0$ is finite. 
If $^-:K\rightarrow K$ is an involution on $K$, then it is straightforward to see that the map $\ast$ given by \[\big(\sum_{i=1}^n k_i\alpha\beta^* \big)^*=\sum_{i=1}^n \overline{k_i}\beta\alpha^* \]
defines the involution on $L_K(E)$ making it into a $\ast$-algebra.
The \textit{canonical} grading given to a Leavitt path algebra is a $\Z$-grading with the
$n$-component
\[L_K(E)_n =\Big\{\sum_ik_i\alpha_i\beta_i^* \Big| \alpha_i,\beta_i\in \textrm{Path}(E), k_ i \in K, \textrm{ and } |\alpha_i | - |\beta_i| = n \textrm{ for all } i\Big\}.\]

\subsection{The Steinberg algebra model of a Leavitt path algebra}
We recall the construction of a groupoid $\mathcal{G}_E$ from an arbitrary graph $E$, which
was introduced in \cite{Kumijan1997} for row-finite graphs and generalized to arbitrary graphs
in \cite{Paterson2002}. We use the notation of \cite{Clark2019}. Define
\[X := E^{\infty}\cup \{\mu \in \textrm{Path}(E) \mid r(\mu) \in \textrm{Sink}(E)\} \cup \{\mu\in \textrm{Path}(E) \mid r(\mu) \in \textrm{Inf}(E)\},\]
and
\[\mathcal{G}_E := \{(\alpha x, |\alpha| - |\beta|, \beta x) \mid \alpha,\beta\in  \textrm{Path}(E), x \in X, r(\alpha) = r(\beta) = s(x)\}.\]
A pair of elements in $ \mathcal{G}_E  $ is  composable if and only if it is of the form
$((x, k, y), (y, l, z))$ and then the composition and inverse maps are defined such
that
\[(x, k, y)(y, l, z) := (x, k + l, z) \hspace*{1cm} \text{~and~}\hspace*{1cm} (x, k, y)^{-1} := (y,-k, x).\]
Thus
$ \mathcal{G}_E^{(0)}= \{(x, 0, x) \mid x \in X\}$, which we identify with $X$.
Next we see how $ \mathcal{G}_E  $ can be viewed as an ample groupoid. For $\mu\in \textrm{Path}(E)$
define
\[Z(\mu) := \{\mu x \mid x \in X, r(\mu) = s(x)\} \subseteq X.\]
For $\mu\in \textrm{Path}(E)$ and a finite $F \subseteq s^{-1}(r(\mu))$, define
\[Z(\mu \setminus F) := Z(\mu) \cap \big(\mathcal{G}_E^{(0)}\setminus(\bigcup_{\alpha\in F}Z(\mu\alpha)\big).\]
The sets of the form $Z(\mu\setminus F)$ are a basis of compact open sets for a
Hausdorff topology on $X =  \mathcal{G}_E^{(0)}$ by \cite[Theorem 2.1]{Webster2014}.

For each $\mu,\upsilon \in \textrm{Path}(E)$ with $r(\mu) = r(\upsilon)$, and finite $F \subseteq \textrm{Path}(E)$ such that $r(\mu) = s(\alpha)$ for all $\alpha\in F$, define
\[Z(\mu,\upsilon) := \{(\mu x, |\mu| - |\upsilon|, \upsilon x) \mid x \in X, r(\mu) = s(x)\},\]
and then
\[Z((\mu, \upsilon) \setminus F) := Z(\mu,\upsilon) \cap \big(\mathcal{G}_E \setminus (\bigcup_{\alpha\in F}
Z(\mu\alpha, \upsilon\alpha)\big).\]
The collection $ Z((\mu, \upsilon) \setminus F) $ forms a basis of compact open bisections that
generates a topology such that $  \mathcal{G}_E $ is a Hausdorff ample groupoid.

Observe that the map $c :  \mathcal{G}_E \rightarrow \Z$ given by $c(x, k, y) = k$ is a continuous
cocyle such that
\[\textrm{Iso}(c^{-1}(0))=c^{-1}(0)\cap \{\gamma\in \mathcal{G}_E\mid s(\gamma)=r(\gamma)  \}=\mathcal{G}_E^{(0)}. \]
Thus $c^{-1}(0)$ is a principal groupoid (and hence an effective groupoid).
 
\cite[Example 3.2]{Clark2015} shows that the map
$\pi: L_R(E)\rightarrow A_R(\mathcal{G}_E)$ such that
\[ \pi(\mu\upsilon^*-\sum_{\alpha\in F}\mu\alpha\alpha^*\upsilon^*)=1_{ Z((\mu, \upsilon) \setminus F) }\]
extends to a \Z-graded $\ast$-isomorphism where for $n \in \Z$
\[A_R(\mathcal{G}_E)_n := \{f \in A_R(\mathcal{G}_E)\mid f(x, k, y) \not=0\Rightarrow k = n\}. \]

Using these results, and Proposition \ref{pro11}, we obtain the following. It is a generalisation of \cite[Proposition 11]{ma2} for Leavitt path algebras over commutative $\ast$-rings.
\begin{corollary}\label{8.2}
  Let $E$ be an arbitrary graph and $R$ a commutative unital $\ast$-ring. Then the following are equivalent.
\begin{enumerate}
\item\label{pro11_1} The involution on $R$ is proper;
\item\label{pro11_2} The involution on $L_R(E)$ is graded proper;
\item\label{pro11_3} The involution on $L_R(E)$ is graded semiproper.
\end{enumerate}
In particular, if $K$ is a field with involution, then the involution on $L_R(E)$ is graded proper.
\end{corollary}

Following \cite[Definition 3.2]{Clark2017} for a  hereditary and saturated subset $H$ we define
\[U_H := \{x \in  \mathcal{G}_E^{(0)}\mid
 r(x_n) \in H \text{~for some~}  n \geq 0\}.\]
 By \cite[Lemma 3.4]{Clark2017} $U_H$ is an open subset of $\mathcal{G}_E^{(0)}$.


\begin{proposition}\label{pro2}
Let $E$ be a finite graph, and $H$ be a hereditary and saturated subset of $E^0$. Then
\begin{enumerate}
\item\label{pro2_1}$U_{H}$ is compact if and only if $F_E(H)$ is finite.
\item\label{pro2_2} $\mathrm{Int}(\mathcal{G}_E^{(0)}\setminus U_{H})$ is compact if and only if $F_E(H^{tc})$ is finite.
\end{enumerate}
\end{proposition}
\begin{proof}
\ref{pro2_1} Follows from \cite[Proposition 2.3]{Clark2017}.

\ref{pro2_2} Using Part \ref{pro2_1} by replacing $H$ with $H^{tc}$, we obtain that
$F_E(H^{tc})$ is finite if and only if $U_{H^{tc}}$ is compact. Thus to prove \ref{pro2_1},  we only need to show that $U_{H^{tc}}=\mathrm{Int}(\mathcal{G}_E^{(0)}\setminus U_{H})$. First, we show $U_{H^{tc}}\subseteq\mathrm{Int}(\mathcal{G}_E^{(0)}\setminus U_{H})$. Note that
 \cite[Lemma 2.1]{Clark2017} implies that
 \begin{align*}
U_{H^{tc}}&=\big( \bigcup_{v\in H^{tc}}Z(v) \big)\cup \big( \bigcup_{\alpha\in F_E(H^{tc})}Z(\alpha) \big). \\
\end{align*}
Thus we need to show that $ Z(v)\subseteq \mathrm{Int}(\mathcal{G}_E^{(0)}\setminus U_{H})$
for each $ v\in H^{tc} $, and $ Z(\alpha)\subseteq \mathrm{Int}(\mathcal{G}_E^{(0)}\setminus U_{H})$ for each $\alpha\in F_E(H^{tc}) $.
Fix $v\in H^{tc}$, and let $x\in Z(v)$. Then $s(x)=v\in H^{tc}$, so  $r(x_i)\not\in H$ for each $i$. This implies  $x\not\in U_H$, and that $x\in \mathcal{G}_E\setminus U_{H}$. Thus
$Z(v)\subseteq \mathcal{G}_E^{(0)}\setminus U_{H}$, and since $Z(v)$ is open we obtain that $Z(v)\subseteq \mathrm{Int}(\mathcal{G}_E^{(0)}\setminus U_{H})$. Now, fix $\alpha\in F_E(H^{tc})$, and let $\alpha x\in Z(\beta)$. Then $s(\alpha),r(\alpha_{i})\not\in H^{tc}$ for $i<|\alpha|$, and  $r(\alpha_{|\alpha|})=s(x)\in H^{tc}$. Thus $r(x_i)\not\in H$  for each $i$, and that $s(\alpha),r(\alpha_{i})\not\in H$ for $i<|\alpha|$.
Thus $\alpha x\not\in U_H$, and so $\alpha x\in \mathcal{G}_E\setminus U_{H}$. Thus
$Z(\alpha)\subseteq \mathcal{G}_E^{(0)}\setminus U_{H}$, and since $Z(\alpha)$ is open, $Z(\alpha)\subseteq \mathrm{Int}(\mathcal{G}_E^{(0)}\setminus U_{H})$. Hence we conclude that 
$U_{H^{tc}}\subseteq \mathrm{Int}(\mathcal{G}_E^{(0)}\setminus U_{H})$.
For the reverse inclusion, let $x\in \mathrm{Int}(\mathcal{G}_E^{(0)}\setminus U_{H})$. Then $s(x)\not\in H$, and for every
initial subpath $\alpha$ of $x$, $\alpha\not\in F_E(H)$. If $s(x)\in H^{tc}$, then $x\in Z(s(x))$, and that $x\in U_{H^{tc}}$. Otherwise, we claim that there is a subpath $\alpha$ of $x$ such that $\alpha\in F_E(H^{tc})$. Towards a contradiction, suppose that  $r(x_i)\not \in H^{tc}$ for each $i$. As  $x\in \mathrm{Int}(\mathcal{G}_E^{(0)}\setminus U_{H})$, there is 
a basic open set $Z(\alpha)$ such that
$x\in Z(\alpha)\subseteq \mathcal{G}_E^{(0)}\setminus U_{H}$.
This yields that $r(\alpha)\not\in H^{tc}$. Then $T(r(\alpha))\cap H\not= \emptyset$. So there is $w\in H$ such that $v\geq w$. Then there exists $\beta\in \textrm{Path}(E)$ such that $s(\beta)=r(\alpha)$ and $r(\beta)=w$. Thus $\alpha\beta\in Z(\alpha)$ and $\alpha\beta\in U_H$, a contradiction.
Hence there is a subpath $\alpha$ of $x$ such that $\alpha\in F_E(H^{tc})$. Then $x\in Z(\alpha)$, and that $x\in U_{H^{tc}}$. Therefore, $\mathrm{Int}(\mathcal{G}_E^{(0)}\setminus U_{H})\subseteq U_{H^{tc}}$ as desired.
\end{proof}
Now we have the following result about the characterization of graded quasi-Baer $\ast$ Leavitt path algebras.
\begin{corollary}[\cite{ma2}, Corollary 20, and \cite{dozd}, proposition 4.5]\label{8.4}
Let $E$ be a finite graph  and $K$ a field with involution.  Then the following are equialent.
\begin{enumerate}
	\item\label{Last_1} The Leavitt path algebra $L_K(E)$ is a graded quasi-Baer $\ast$-ring;
	\item\label{Last_2} For each nonempty hereditary and saturated subset $H$ of $E^0$,   no cycle outside $H$ leads to $H$ or no cycle outside $H^{tc}$ leads to $H^{tc}$;
    \item\label{Last_2'} For each nonempty hereditary and saturated subset $H$ of $E^0$, no vertex of a cycle of $E$ with vertices in $E^ 
    0\setminus (H \cup H^{tc})$ emits a path to $H$ and a path to $H^{tc}$.
	\item\label{Last_3}  For each nonempty hereditary and saturated subset $H$ of $E^0$, the saturated closure of $H^{\perp}\cup H^{\perp\perp}$ is  $E^0$;
	\item\label{Last_3'} For each nonempty hereditary and saturated subset $H$ of $E^0$, $H^{\perp}\vee H^{\perp\perp}=E^0$.
	\end{enumerate} 
\end{corollary}
\begin{proof}
	\ref{Last_1}$\Leftrightarrow$\ref{Last_2}
	Note that when $E$ is a   finite graph,  $F_E(H)$ is finite if and only if  no cycle outside $H$ leads to $H$ or no cycle outside $H^{tc}$ leads to $H^{tc}$. Now
Proposition \ref{pro2} and Theorem \ref{1}  by considering the graph groupoid $\mathcal{G}_E$ yield the result.

\ref{Last_2}$\Leftrightarrow$\ref{Last_2'} It is clear.

\ref{Last_2'}$\Rightarrow$\ref{Last_3} We show that if \ref{Last_2'} holds and \ref{Last_3} fails, we
arrive to a contradiction. 
Assume that \ref{Last_3} fails for some $H$ and let $S$ denote
the saturated closure of $H^{\perp}\cup H^{\perp\perp}$. So, there is $v_0 \in E^0\setminus S$. If $v_0$
were a sink, then $v_0 \not\in H^{\perp}$ would imply that $v_0\in  H^{\perp\perp}\subseteq S$. As $v_0 \not\in S$, $v_0$
necessarily emits some edges. Their ranges are not all in $S$ because $v_0 \not\in S$
and $S$ is saturated, so at least one of these edges, say $e_1$, has its range outside
of $S$. If $r(e_1) = v_0$, $c = e_1$ is a cycle. If $v_1 = r(e_1) \not= v_0$, one can use the same
argument as the one used to show that $v_0$ emits edges to conclude that $v_1$
emits edges (if $v_1$ were a sink, then $v_1 \not\in  H^{\perp}$ implies that $v_1 \in H^{\perp\perp}$ and that
would contradict $v_1 \not\in S$). So, there is $e_2$ which $v_1$ emits. If $r(e_2) = v_0$, then
$e_1e_2$ is a closed path with all of its vertices outside of $S$, so there is also a
cycle $c$ which contains $v_0$ with all its vertices outside of $S$. If $v_2 = r(e_2) \not= v_0$,
we continue the process. As $E^0$ is finite, this process eventually ends and we
arrive to a cycle $c$ which contains $v_0$ and having all its vertices outside of $S$. The cycle $c$ has exits since otherwise $v_0 \not\in  H^{\perp}$ would imply that all vertices of $c$ are in $H^{\perp\perp}\subseteq S$. By the assumption that \ref{Last_2} holds, the ranges of these exits either all connect to $H$ or all connect to $H^{tc}(=H^{\perp})$. In the first case, the vertices of $c$ would be in $H^{\perp\perp}$. In the second case, the vertices of $c$ would be in $H^{\perp}$.
In each case, the vertices of $c$ end up being in $S$ which is a contradiction.

\ref{Last_3}$\Rightarrow$\ref{Last_2'} We show that if \ref{Last_3} holds and \ref{Last_2'} fails, then we arrive to a contradiction. If \ref{Last_2'} fails for some $H$, let $c$ be a cycle whose existence the failure of \ref{Last_2'} guarantees. As $c$ emits paths to $H$, the vertices
of $c$ are not in $H^{\perp}$. As $c$ emits paths to $H^{\perp}$, the vertices of $c$ are not in $H^{\perp\perp}$.
Thus, the vertices of $c$ are outside of $H^{\perp}\cup H^{\perp\perp}$. As \ref{Last_3} holds, the vertices of $c$ are in the saturated closure of $H^{\perp}\cup H^{\perp\perp}$ which implies that any path from $v_0 = s(c)$ connects to $H^{\perp}\cup H^{\perp\perp}$. Thus, for $v_0$ there is $n$ such that
 $v_0\in (H^{\perp}\cup H^{\perp\perp})_n$, where 
 \begin{align*}
 	(H^{\perp}\cup H^{\perp\perp})_0&=H^{\perp}\cup H^{\perp\perp}, \text{~and~}\\
 	(H^{\perp}\cup H^{\perp\perp})_n&=(H^{\perp}\cup H^{\perp\perp})_{n-1}\cup \{v\in \textrm{Reg}(E)\mid \{r(e)\mid s(e)=v\}\subseteq  (H^{\perp}\cup H^{\perp\perp})_{n-1}\}
 \end{align*}
 This means that the range $v_1$ of the edge $e_0$ which $v_0$ emits in $c$ is in $(H^{\perp}\cup H^{\perp\perp})_{n-1}$. Repeating this argument for $v_1$, we arrive to $v_2$ in $c$ which is in $(H^{\perp}\cup H^{\perp\perp})_{n-2}$. Eventually, we arrive to $v_n$ in $c$ which is in $(H^{\perp}\cup H^{\perp\perp})_0=H^{\perp}\cup H^{\perp\perp}$. Since $H^{\perp}\cup H^{\perp\perp}$ is hereditary, this implies that
 every vertex of $c$ is in $H^{\perp}\cup H^{\perp\perp}$. This is a contradiction since $c$ was chosen
 with vertices outside of $H^{\perp}\cup H^{\perp\perp}$.

\ref{Last_3}$\Leftrightarrow$\ref{Last_3'} It is evident since $E$ is finite.
\end{proof}

\medskip

\noindent{\bf Funding:}
The work of the first-named author  was partially supported by a grant (No. 99033216) from Iran National Science Foundation (INSF).

\vskip4.0pc

\bibliographystyle{amsplain}

\end{document}